\documentclass[11pt]{article}
\usepackage{amsmath}
\usepackage{amssymb}
\usepackage{amsthm}
\usepackage{mathtools}

\newtheorem{theorem}{Theorem}
\newtheorem{prop}[theorem]{Proposition}
\newtheorem{lemma}[theorem]{Lemma}
\newtheorem{conj}{Conjecture}
\newtheorem{cor}[theorem]{Corollary}

\newcommand{\ba}{\mathbf{a}}

\newcommand{\bb}{\mathbf{b}}

\newcommand{\be}{\mathbf{e}}
\newcommand{\bi}{\mathbf{i}}

\newcommand{\br}{\mathbf{r}}
\newcommand{\bs}{\mathbf{s}}
\newcommand{\bv}{\mathbf{v}}
\newcommand{\bw}{\mathbf{w}}
\newcommand{\bx}{\mathbf{x}}

\newcommand{\bE}{\mathbf{E}}

\newcommand{\CC}{\mathbb{C}}

\newcommand{\II}{\mathbb{I}}
\newcommand{\JJ}{\mathbb{J}}
\newcommand{\NN}{\mathbb{N}}

\newcommand{\RR}{\mathbb{R}}
\newcommand{\ZZ}{\mathbb{Z}}

\DeclareMathOperator\spec{Spec}
\allowdisplaybreaks

\title{Adjacency Spectra of Random and Complete Hypergraphs}
\author{Joshua Cooper}
\begin{document}

\maketitle

\begin{abstract} We present progress on the problem of asymptotically describing the complex homogeneous adjacency eigenvalues of random and complete uniform hypergraphs.  There is a natural conjecture arising from analogy with random matrix theory that connects these spectra to that of the all-ones hypermatrix.  Several of the ingredients along a possible path to this conjecture are established, and may be of independent interest in spectral hypergraph/hypermatrix theory.  In particular, we provide a bound on the spectral radius of the symmetric Bernoulli hyperensemble, and that the spectrum of the complete \(k\)-uniform hypergraph for \(k=2,3\) is close to that of an appropriately scaled all-ones hypermatrix.
\end{abstract}

\section{Introduction}

Spectral hypergraph theory attempts to link the spectral structure of matrices and hypermatrices\footnote{There is a convention in the literature to refer to hypermatrices as ``tensors'', although the identification is not entirely consistent with other fields' use of these terms.  Other use simply ``matrix.''  We choose to use ``hypermatrix'' to refer to a multidimensional array in order to avoid any confusion on the matter.} associated with hypergraphs (i.e., set systems, usually in which every set has the same cardinality -- termed ``uniform'') to the combinatorial properties of those hypergraphs in a manner analogous to spectral graph theory.  It is far from straightforward to define appropriate hypermatrices and spectra of such objects, however, and the strengths and weaknesses of each theory are a subtle matter.  The main approaches to date can be roughly classified as follows:

\begin{enumerate}
\item {\bf Matrices and linear spectral theory}.  One natural tactic is to define various types of matrices associated with the hypergraph and relate their ``ordinary'' eigenvalues back to the hypergraph.  Prominent successes in this vein include Feng-Li \cite{MR1405722} and Mart\'{i}nez-Stark-Terras \cite{MR1814089}, who used vertex-edge incidence matrices; Lu-Peng \cite{MR2993134,MR3021338}, who defined a kind of higher-order adjacency matrix between sets of vertices instead of single vertices; and Rodriguez \cite{MR1890984,MR1995660,MR2523606}, who also defined a type of adjacency matrix by counting edges containing vertex pairs.
\item {\bf Hypermatrices and low-rank decompositions}.  Another approach is to define a {\it hypermatrix}, i.e., a higher (than two) dimensional array of numbers whose entries (for example) indicate whether the indices of that entry are a hyperedge; a natural choice of indicator is to use 0 and 1 as in the case of graphs.  The dimensionality of the array is known as the hypermatrix's ``order''; it is called ``cubical'' if the set of indices in each such dimension has equal cardinality; and it is said to have ``dimension \(n\)'' if this cardinality is \(n\).  Then, one asks for a decomposition into rank one hypermatrices (outer products of vectors with themselves), just as in the classical Spectral Theorem from linear algebra.  While low-rank decomposition of tensors (for present purposes, multilinear maps that are associated with hypermatrices under particular choices of basis) are a highly studied subtopic of applied multilinear algebra -- ``Tucker'' and ``canonical polyadic'' decompositions are umbrella terms for a number of such techniques -- it is not easy to see how to use them for adjacency hypermatrices.  Nonetheless, Friedman and Wigderson \cite{MR1325271} did essentially that with considerable success, later spawning deeper results such as those of Keevash-Lenz-Mubayi \cite{1304.0050, 1208.4863, 1208.5978, 1309.3584}, who have been able to connect quasirandomness, expansion properties, and ``second largest eigenvalues'' in a way that closely resembles Cheeger inequalities and the Expander Mixing Lemma of classical spectral graph theory.   Another, highly algebraic approach worth noting here is that described by Gnang, Elgammal, and Retakh (e.g., in \cite{MR2918215}).  Their strategy uses a variant of Bhattacharya-Mesner tensor algebra to define a different kind of spectral decomposition for hypermatrices.
\item {\bf Hypermatrices and polynomial maps}.  One can define hypermatrices associated with a hypergraph as above, but consider their entries as coefficients of multilinear forms and study the algebraic and analytic properties of the maps given by these polynomials and the homogeneous polynomials of which they are polarizations (when the hypermatrices are fully symmetric, i.e., permutation of index sets fixes the entries).  This approach to studying symmetric tensors dates back at least to Lyusternik-{\v{S}}nirel'man \cite{MR0029532}, but was resurrected by Lim \cite{Lim_singularvalues} in an analytic form and Qi \cite{MR2178089} in an equivalent but more algebraic fashion; in the latter case, the author gave a very useful (and intuitively appealing) description of the spectrum as the roots of a characteristic polynomial, the resultant of the polynomial system whose coefficients are drawn from the original hypermatrix with \(\lambda\) subtracted from its diagonal entries.  The subject one obtains by applying this definition of spectra of hypermatrices to adjacency (and related, such as Laplacian) hypermatrices has become a very rapidly growing field: see, for example, \cite{MR3163144,MR2996375,MR2435198,MR2476927,MR3141883,MR2900714,MR2996365,MR2472189,MR2974300,MR2900714,MR3116407,MR3175063,MR3230456,MR3194370,MR3148904,MR3141891,MR3091297,MR3091308,MR3181972,MR2854611,MR3141884}.  In work most relevant to the present manuscript, the complex homogeneous eigenpairs of the adjacency hypermatrix \(\mathcal{A}(H)\) are the object of study, i.e., solutions \((\lambda,\mathbf{v})\) to the equation
\[
\mathcal{A}(H) : \mathbf{v}^{\otimes k-1} = \lambda \mathbf{v}^{\circ (k-1)},
\]
where ``\(:\)'' denotes tensor contraction, ``\(\otimes\)'' tensor product, and ``\(\circ\)'' denotes the product of hypermatrices known variously as ``Hadamard'', ``Schur'', and ``entrywise''.  (See the next section for more on our choice of notation.)  For a \(k\)-uniform hypergraph \(H\), the left-hand side denotes ``multiplication on all but one side'' of the order \(k\) hypermatrix \(\mathcal{A}(H)\) by a vector, and the right-hand side denotes a scalar multiple of the same vector with all entries raised to the \((k-1)\)-st power.
\end{enumerate}

One of the most applicable topics in spectral graph theory is the theory of the spectrum of random graphs; this area serves as a crucial tool for understanding quasirandomness, graph expansion, and mixing time of Markov chains, for example.  A natural desideratum, therefore, is a description of the spectra of random (Erd\H{o}s-R\'{e}nyi) hypergraphs, i.e., hypergraphs in which edges are chosen according to iid Bernoulli distributions.  One very special case of this problem, when the probability of including an edge is \(1\), is the question of describing the spectrum of the complete graph.  In the case of graphs, a complete description of the eigenvalues (and eigenvectors) of complete hypergraphs is an easy exercise, but the same task is surprisingly difficult for hypergraphs.  For hypergraphs of uniformity \(3\), \cite{MR2900714} gives a kind of implicit description that is not very illuminating and does not seem to generalize well.  For higher uniformities, the adjacency hypermatrix is very close to the (scaled) all-ones hypermatrix, the eigenpairs of which were completely described in \cite{1301.4590}, differing only in the vanishing fraction of entries corresponding to repeated indices (i.e., degenerate edges).  Therefore, a theorem to the effect that, if two hypermatrices are ``close'' in some sense, their spectra are close as well would suffice to understand the spectrum of complete hypergraphs.  We venture the following conjecture:

\begin{conj} \label{conj:Weyl-like} Suppose \(A\) and \(B\) are hypermatrices so that \(\|A-B\| \leq \epsilon\) for some norm \(\|\cdot\|\) (or spectral radius) and \(\epsilon > 0\).  There there is a bijection \(\rho\) between the eigenvalues (with multiplicity) of \(A\) and \(B\) and a function \(f\) with \(\lim_{\epsilon \rightarrow 0} f(\epsilon) = 0\) so that \(|\lambda - \rho(\lambda)| < f(\epsilon)\) for each eigenvalue \(\lambda\).
\end{conj}

This would imply, among other results, that that the eigenvalues of a complete hypergraph come close to the eigenvalues of an all-ones hypermatrix, ignoring multiplicities; below, we actually prove this for \(k=2\) and \(k=3\).  
Indeed, inequalities of this flavor -- the Weyl, Hoffman-Wielandt, and Lidskii inequalities most prominently -- are an important topic in matrix theory.  The standard proofs of such statements, however, rely on the Spectral Theorem via the Courant-Fischer Theorem, two key tools which are as of yet (mostly) missing for hypermatrices.  However, there are ``variational'' proofs that do not rely on the Spectral Theorem (see, for example, \cite{MR1335452,MR2906465}); the idea is to show that the eigenvalues vary (piecewise) smoothly along a curve in matrix space connecting \(A\) and \(B\) and that the derivative of the eigenvalues are bounded by some norm of \(A-B\), so that one can integrate along the path to obtain a perturbational result.  Analogizing such proofs for hypermatrices is far from straightforward but should be possible.  For example, \cite{LiQiYu} provides a smoothness result for spectral radii.  In the present paper, we apply results from analytic function theory (see \cite{MR2806697,1309.2151}) to show that all eigenvalues do indeed vary piecewise smoothly along line segments connecting \(A\) and \(B\) in hypermatrix space.  The step of showing that the derivative is bounded by a norm-like function of \(A-B\) is the next main problem to address in this program, and is made difficult both by the fact that matrix-theoretic proofs of analogous results use smoothness of the eigenvectors (something that is difficult to generalize to hypermatrices because the singular surfaces involved are no longer linear spaces but some more complicated projective varieties) and ordinary norms are not in general holomorphic (because they involve the modulus function).  However, there is reason to believe that the relatively simple differential behavior of algebraic sets along curves will permit the effective use of homotopy techniques (tools such as \cite{MR1294431,MR2219034,MR2806697}).

One reason to push for a generalization of perturbational inequalities such as the conjecture above is that it would also lead to a full (asymptotic almost sure) description of the eigenvalues of random hypergraphs.  Indeed, the adjacency hypermatrix of a random hypergraph is, with high probability, very close in any reasonable norm to another appropriately scaled copy of the all-ones hypermatrix, so that the above program would also yield the distribution of their spectra.  Since the eigenvalues of random graphs (and the closely related topic of random \(0/1\) and random \(\pm 1\) matrix ensembles of various types) is another central and important area in matrix theory, this is a natural problem to consider.  Below, we show that random \(\pm 1\) symmetric hypermatrices have small spectral radius (maximum modulus of an eigenvalue) with high probability in a certain precise sense, so again the missing piece of this argument is a Weyl-type inequality.

In the next section, we introduce notation and definitions used throughout, as well as outline an argument approaching Conjecture \ref{conj:Weyl-like}.  In Section \ref{sec:complete}, we show that the difference between a suitably scaled all-ones hypermatrix and the adjacency hypermatrix of a complete hypergraph has a small spectral radius, and then prove that the eigenvalues of these two hypermatrices are indeed close (in Minkowski set distance) for \(k=2\) and \(k=3\).  In Section \ref{sec:random}, we carry out similar analysis for random hypergraphs.  Then, in Section \ref{sec:smootheigs}, we show that the eigenvalues along a line segment of hypermatrices vary piecewise smoothly.

\section{Definitions}

For a hypergraph \(H = (V(H),E(H))\) of uniformity \(k\) (i.e., \(E(H) \subset \binom{V(H)}{k}\)), we define its {\it adjacency hypermatrix} as the hypermatrix \(A_H = [a_{i_1,\ldots,i_k}]\) whose \((i_1,\ldots,i_k)\)-th entry is given by
\[
a_{i_1,\ldots,i_k} = \frac{1}{(k-1)!} \left \{ \begin{array}{ll} 1 & \text{ if } \{i_1,\ldots,i_k\} \in E(H) \\ 0 & \text{ otherwise.} \end{array} \right .
\]
Here, \(\bi = (i_1,\ldots,i_k)\) is referred to as a {\it multiindex}.  A hypermatrix is {\it symmetric} if its entries are invariant under permutation of the entries of its multiindices.  The {\it shape} of a hypermatrix is the vector of lengths of its index sets; the dimension of this vector is the {\it order} of the hypermatrix.  If all entries of the shape are equal a hypermatrix is {\it cubical}, and its common index set size is the {\it dimension}; note that the adjacency hypermatrix of a hypergraph is cubical of dimension equal to the number of vertices.  Given two hypermatrices \(A\) and \(B\) of the same shape, write \(A \circ B\) for the {\it Hadamard product}, i.e., the hypermatrix of this same shape whose entries are given by
\[
(A \circ B)_{\bi} = A_{\bi} B_{\bi}
\]
where \(\bi\) is a multiindex.  Write \(\otimes\) for {\it tensor product}, i.e., given \(A\) and \(B\) hypermatrices of shapes \(\ba = (a_1,\ldots,a_s)\) and \(\bb = (b_1,\ldots,b_t)\), for \(\bi = (c_1,\ldots,c_{s+t})\) with \(c_j \in [a_j]\) for \(j \in [s]\) and \(c_j \in [b_{j-s}]\) for \(j \in [s+t] \setminus [s]\),
\[
(A \otimes B)_{\bi} = A_{c_1,\ldots,c_s} B_{c_{s+1},\ldots,c_{s+t}}.
\]
In general, given an (associative) binary operator \(\ast\) and an object \(X\) to which it is applied, we define \(X^{\ast 1} = X\) and, inductively, \(X^{\ast k} = X \ast X^{\ast (k-1)}\) for integers \(k > 1\).  Write \(\JJ_n^k\) for the cubical order \(k\) hypermatrix of dimension \(n\) all of whose entries are \(1\), i.e., \(\JJ_n^k = \hat{1}^{\otimes k}\), where \(\hat{1}\) is the all-ones vector of dimension \(n\); if \(k=1\) we omit the superscript; write \(\II_n^k\) for the cubical order \(k\) identity hypermatrix of dimension \(n\), i.e., all of whose diagonal entries (all entries of the multiindex equal) are \(1\) and \(0\) otherwise.  Given two hypermatrices \(A\) and \(B\) with \(A\) order \(k\) and a set of indices \(S \subset [k]\) with \(|S| = s\) and \(B\) of order \(s\), write \(A :_{S} B\) for {\it tensor contraction}:
\[
(A :_{S} B)_{\bi=i_1,\ldots,i_{k-s}} = \sum_{a_1,\ldots,a_s = 1}^n A_{\bi(a_1,\ldots,a_s;S)}B_{a_1,\ldots,a_s}
\]
where \(\bi(a_1,\ldots,a_s;S)\), with \(S = \{j_1,\ldots,j_s\}\), denotes the multiindex whose \(t\)-th coordinate is \(i_r\) if \(t \not \in S\) and \(t\) is the \(r\)-th largest element of \([n] \setminus S\) or is \(j_r\) if \(t \in S\) and \(t\) is the \(r\)-th largest element of \(S\).  If \(S = [k]\) or \(A\) is symmetric, then we omit the subscript on ``\(:\)'' and simply write \(A : B\) (if the shape of \(A\) and \(B\) are the same, this reduces to the Frobenius product of \(A\) and \(B\); thus, the choice of notation).  Finally, in a slight abuse of notation, we write \(A^{: t}\) for \(\JJ_n^k : A^{\circ t} = A^{\circ (t-1)} : A\).

An {\it eigenpair} of a hypermatrix \(A\) is a pair \((\bv,\lambda) \in \RR^n \times \RR\) with \(\bv \neq \hat{0}\) so that
\[
A : \bv^{\otimes k-1} = \lambda \bv^{\circ k-1}
\]
By \cite{MR2178089}, in close analogy to linear algebra, it is also possible to define the eigenvalues (the \(\lambda\) above) as the roots of the equation
\[
\det(\lambda \II_n^k - A) = 0
\]
where \(\det(B)\) is the {\it hyperdeterminant} of hypermatrix \(B\), defined as the resultant of the polynomial coordinates of \(B : \bx^{\otimes k}\), where \(\bx\) is a vector of variables.  (The definition of hyperdeterminant/resultant is difficult to give explicitly; we refer to \cite{MR2394437} for the interested reader.)  

In the interest of approaching Conjecture 1, we make the following argument.  Suppose that \(A = A(t)\) is a smooth hypermatrix-valued function, \(t \in [0,1]\), with the property that it is possible to assign indices to the \(n (k-1)^n\) eigenpairs of \(A(t)\) so that \((\bv_j(t),\lambda_j(t))\) is a smooth function of \(t\).  Then we may write
\begin{equation} \label{eq1}
A : \bv^{\otimes k-1} = \lambda \bv^{\circ k-1}.
\end{equation}
Assume for the moment that
\begin{equation} \label{eq2}
\bv^{: k} = 1
\end{equation}
as well.  This assumption loses no generality, because eigenvectors can be freely scaled, as long as \(\bv^{: k} \neq 0\); otherwise, a different argument must be made.  Differentiating (\ref{eq2}) (and writing the derivative of \(f(t)\) with respect to \(t\) as \(f^\prime\)), we obtain
\begin{equation} \label{eq3}
\bv^{\circ k-1} : \bv^\prime + \bv^{\circ k-2} \circ {\bv^\prime} : \bv + \bv^{\circ k-3} \circ {\bv^\prime} : \bv^{\circ 2} + \cdots = 0.
\end{equation}
Note, however, that all addends on the left-hand side are equal, so they are all zero.  Operate on both sides of (\ref{eq1}) on the right by ``\(: \bv\)'' to obtain
\begin{equation} \label{eq4}
A : \bv^{\otimes k} = \lambda \bv^{: k} = \lambda.
\end{equation}
Then, differentiating (\ref{eq4}), we obtain
\begin{align*}
{\lambda^\prime} &= \frac{d}{dt} \left ( A : \bv^{\otimes k} \right ) \\
&= {A^\prime} : \bv^{\otimes k} + A : ({\bv^\prime} \otimes \bv^{\otimes k-1}) +  A : (\bv \otimes {\bv^\prime} \otimes \bv^{\otimes k-2}) + \cdots \\
&= {A^\prime} : \bv^{\otimes k} + k A : (\bv^{\otimes k-1} \otimes {\bv^\prime}) \\
&= {A^\prime} : \bv^{\otimes k} + k A : \bv^{\otimes k-1} : {\bv^\prime} \\
&= {A^\prime} : \bv^{\otimes k} + k \lambda \bv^{\circ k-1} : {\bv^\prime} \\
&= {A^\prime} : \bv^{\otimes k},
\end{align*}
by symmetry of \(A\) and (\ref{eq3}).  Continuing, we have
\[
{\lambda^\prime} = {A^\prime} : \bv^{\otimes k} \leq \sup_{\bv \in \Sigma} \left ( {A^\prime} : \bv^{\otimes k} \right ) \eqqcolon \left \| {A^\prime} \right \|_\Sigma,
\]
where \(\Sigma \subset \CC^n\) is some set containing all eigenvectors \(\bv_j(t)\) for all \(t \in [0,1]\).  Therefore,
\[
\left | \lambda_j(0) - \lambda_j(1) \right | \leq \int_0^1 \left \| {A^\prime}(t) \right \|_\Sigma \, dt
\]
In particular, if \(A(t) = A_0 + Bt\) and define \(A_1 = A(1)\), we obtain
\begin{equation} \label{eq5}
\left | \lambda_j(A_0) - \lambda_j(A_1) \right | \leq \| B \|_\Sigma.
\end{equation}
Even if the \(\lambda_j\) and \(\bv_j\) are only piecewise smooth, this argument still provides the same upper bound by telescoping the sum of the piecewise integrals obtained by breaking the \(t\)-path at singular points.  If a reasonable set \(\Sigma\) can be chosen -- in particular, so that \(\|B\|_\Sigma\) is small -- the result is a Weyl-type inequality and thus a description of the spectrum of the complete and random hypergraphs.  In Section \ref{sec:smootheigs}, we show that the eigenvalues of \(A(t)\) indeed vary piecewise smoothly in \(t\), so the problem remains to show that eigenvectors do as well.

\section{Complete Hypergraphs}
\label{sec:complete}

Let \(B(n,k)\) denote the difference between the appropriately scaled all-ones hypermatrix and the adjacency matrix of the complete \(k\)-uniform hypergraph, i.e.,
\[
B(n,k) = \frac{1}{(k-1)!} \mathbb{J}_n^k - \mathcal{A}(K_n^k).
\]
For a hypermatrix \(A\), define the spectral radius of \(A\) by
\[
\rho(A) = \max_{\lambda \in \spec(A)} |\lambda|.
\]
We apply the following lemma from \cite{MR2685169}.

\begin{lemma} \label{lem:YY} Suppose \(A\) is a nonnegative order \(k\), dimension \(n\) hypermatrix with entries \(a_{i_1\cdots i_k}\).  Then
\[
\min_{j} \sum_{i_2,\ldots,i_k \in [n]} a_{ji_2\cdots i_k} \leq \rho(A) \leq \max_{j} \sum_{i_2,\ldots,i_k \in [n]} a_{ji_2\cdots i_k}
\]
\end{lemma}

It is simple to obtain a bound on the spectral radius of \(B(n,k)\) from this lemma.

\begin{theorem} For \(n > k\), we have
\[
\rho(B(n,k)) \leq \frac{(n-1)^{k-2}}{(k-2)!}.
\]
\end{theorem}
\begin{proof} The number of \(k\)-tuples \(i_1,\ldots,i_k\) with \(i_1 = j\) and no indices repeated (i.e., \(i_s \neq i_j\) for \(s \neq t\)) is \((n-1)^{\underline{k-1}}\).  Therefore, the number of \(k\)-tuples \(i_1,\ldots,i_k\) with \(i_1 = j\) and some index repeated is
\begin{align*}
(n-1)^{k-1} - (n-1)^{\underline{k-1}} &= \sum_{i=1}^{k-1} \binom{k-1}{i} (n-1)^{k-1-i} \\
&\leq (n-1)^{k-1} \sum_{i=1}^{k-1} \frac{1}{i!} \left( \frac{k-1}{n-1} \right )^{i} \\
&\leq (n-1)^{k-1} \sum_{i=1}^{\infty} \left( \frac{k-1}{n-1} \right )^{i} \\
&\leq (n-1)^{k-1} \frac{(k-1)/(n-1)}{1 - (k-1)/(n-1)} \\
&\leq (n-1)^{k-1} \frac{k-1}{n - k}\\
&\leq (k-1) (n-1)^{k-2}.
\end{align*}
Therefore, by Lemma \ref{lem:YY},
\[
\rho(B(n,k)) \leq \frac{1}{(k-1)!} (k-1) (n-1)^{k-2} = \frac{(n-1)^{k-2}}{(k-2)!}.
\]
\end{proof}

We are now in a position to prove a weak form of the conjectured similarity between eigenvalues of complete hypergraphs and the appropriately scaled all-ones hypermatrix, for \(k \leq 3\).  In particular, this result discards all information about multiplicities.  For two sets \(A, B \subset \mathbb{C}\), recall that the Minkowski distance \(\delta(A,B)\) between the sets is defined as \(\inf \{r : A \subset \mathcal{N}_r(B) \textrm{ and } B \subset \mathcal{N}_r(A)\}\), where \(\mathcal{N}_r(S) = \bigcup_{z \in S} B_r(z)\) and \(B_r(z) = \{w : |z-w| < r\}\).  Equivalently, \(A\) and \(B\) differ by \(\delta\) in Minkowski distance if there are maps \(f: A \rightarrow B\) and \(g : B \rightarrow A\) both of which are pointwise within \(\delta\) of the identity function.

\begin{theorem} \label{thm:k2and3} For \(k=2,3\) and \(n \rightarrow \infty\), the set of eigenvalues \(L\) of the \(k\)-uniform complete graph \(K^{(k)}_n\) on \(n\) vertices and the set of eigenvalues \(M\) of \((k-1)!^{-1} \mathbb{J}^k_n\) satisfy \(\delta(M,L) = o(n^{k-1})\).
\end{theorem}

Therefore, there is an approximate bijection between the set of eigenvalues of \(K^{(k)}_n\) and the set of eigenvalues of \((k-1)!^{-1} \mathbb{J}\).  Note that, by Lemma \ref{lem:YY}, the spectral radii of \(K^{(k)}_n\) and \((k-1)!^{-1} \mathbb{J}^k_n\) are \((k-1)!^{-1} n^{k-1} (1+o(1))\).  Thus, in this result, the order of magnitude of the Minkowski distance is such that only the largest (i.e., modulus \(\Theta(n^{k-1})\)) eigenvalues are meaningfully connected.  We employ the following result from \cite{1301.4590}:

\begin{theorem} \label{thm:charpolyJJ} Let \( \phi_n(\lambda) \) denote the characteristic polynomial of \( \JJ_n^k \), \( n \geq 2 \).  Then
\[
\phi_n(\lambda) = \lambda^{(n-1) (k-1)^{n-1}} \prod_{\substack{\br \in \NN^{k-1} \\ \br \cdot \hat{1} = n}} \left ( \lambda - (\br \cdot \hat{\zeta})^{k-1} \right )^{\mu(\br)/(k-1)}.
\]
where \( \hat{\zeta} \) denotes the vector \( (1,\zeta_{k-1},\ldots,\zeta_{k-1}^{k-2}) \) for any primitive $k$-th root of unity \(\zeta_{k-1}\), \(\hat{1}\) denotes the all-ones vector, and \(\mu(\br)\) is the number of ways to choose \( s_0,\ldots,s_{m-2} \) so that
\[
(\bs \cdot \hat{\zeta})^{k-1} = ( \br \cdot \hat{\zeta} )^{k-1}.
\]
\end{theorem}

We also make use of the following result from \cite{MR2900714}.

\begin{theorem} \label{complete3unif} The complete $3$-uniform hypergraph on $n$ vertices has eigenvalues $0, 1, \binom{n-1}{2}$, and at most $2n$ others, which can be found by substituting the roots of one of $n/2$ univariate quartic polynomials
\begin{align*}
P_t(c) & = \binom{t}{2}c^4 + t(n-t-1)c^3 + \left(\binom{n-t-1}{2} - \binom{t-1}{2}\right) c^2 \\
& \qquad - (t-1)(n-t)c - \binom{n-t}{2}
\end{align*}
for some \(1 \leq t \leq n/2\) into the quadratic polynomial
\begin{equation} \label{comp3uniflambda}
f_t(c) = \binom{t}{2}c^2 + t(n-t-1)c + \binom{n-t-1}{2}.
\end{equation}

\end{theorem}

\begin{proof}[Proof of Theorem \ref{thm:k2and3}] First, we consider \(k=2\).  There, the eigenvalues of the all-ones matrix \(\mathbb{J}^2_n\) are simply \(0\) and \(n\).  On the other hand, the eigenvalues of \(\mathcal{A}(K_n) = \mathbb{J}^2_n - I_n\) are \(-1\) and \(n-1\), and clearly \(0 = -1 + o(n)\) and \(n = n-1 + o(n)\).

Now, suppose \(k=3\).  Throughout the following, we use \(C_j\) to denote absolute constants.  By Theorem \ref{thm:charpolyJJ}, the eigenvalues of \(1/2 \mathbb{J}^3_n\) are \(0\) and one half the roots of 
\[
\phi_n(\lambda) = \prod_{\substack{a,b \geq 0 \\ a+b = n}} \left ( \lambda - (a-b)^2 \right ).
\]
Therefore, any nonzero eigenvalue \(\mu\) is given by \(\mu = \lambda/2 = s^2/2\) for some \(0 \leq s \leq n\) such that \(s\) and \(n\) agree in parity.  Let the set of all such eigenvalues be denoted \(M\).

On the other hand, by Theorem \ref{complete3unif}, the eigenvalues of \(\mathcal{A}(K^3_n)\) are $0, 1, \binom{n-1}{2}$ and a set \(S\) of at most \(2n\) others obtained by plugging roots of the polynomials \(P_t(c)\) into the polynomials \(f_t(c)\), for some \(0 \leq t \leq n/2\), as described above.  Note that \(0\), \(1\), and \(\binom{n-1}{2}\) are within \(o(n^2)\) of elements of \(M\), since \(0 \in M\) and \(n^2/2 \in M\).  

Now, choose \(t \in [0,n/2]\), and let \(\alpha = t/n\).
\begin{align*}
P_t(c) & = (t^2-t) c^4/2 + (nt-t^2-t) c^3 + ( n^2-2nt-3n +6t) c^2/2 \\
& \qquad + (-nt+n+t^2-t)c + (-n^2+2nt-t^2+n-t)/2 \\
& = \frac{n^2}{2} (\alpha^2 c^4 - 2 \alpha^2 c^3 + 2 \alpha c^3 + 2 \alpha^2 c - 2 \alpha c^2 - \alpha^2 - 2 \alpha c + c^2 + 2 \alpha - 1) \\ 
& \qquad - \frac{n}{2} (\alpha c^4 + 2 \alpha c^3 - 6 \alpha c^2 + 2 \alpha c + 3 c^2 + \alpha - 2 c - 1) \\
& = \frac{n^2}{2} (\alpha c - \alpha + 1)^2 (c + 1) (c - 1) \\
& \qquad - \frac{n}{2} (\alpha c^3 + 3\alpha c^2 - 3\alpha c - \alpha + 3 c + 1)(c - 1)
\end{align*}
Note that \(c=1\) is a root, so \(f_t(1) = n^2/2 - 3n/2 + 1
= n^2/2 + o(n^2)\) is an eigenvalue of \(K^{(3)}_n\). It is straightforward to verify that \(P_t(-1) \leq 0\) for all \(\alpha \in [0,1/2]\).  On the other hand, letting \(\epsilon = 3/(n (2\alpha-1))\), we have 
\[
P_t\left (-1+\epsilon \right) = -2(2\alpha-1)n \pm \frac{C_1}{(2\alpha-1)^4}
\]
for some absolute constant \(C_1\), which is positive for \( \alpha < 1/2 - C_2 n^{-1/5}\) (in which case \(\epsilon < C_3 n^{-4/5}\)).  Furthermore, letting \(\beta = (1/2-\alpha ) n^{1/5}\),
\[
P_t\left (-1 - \frac{2}{n^{1/5}} \right) = 2 (2 \beta - 1)^2 n^{7/5} \pm C_4 n^{6/5}
\]
so that, for \(\alpha > 1/2 - C_2 n^{-1/5}\) and \(n\) sufficiently large, this is positive.

Therefore, for any \(\alpha \in [0,1/2]\), \(P_t\) has a root between \(-1-\epsilon\) and \(-1\), say, \(-1 - \delta\), with \(0 \leq \delta < C_6 n^{-1/5}\), and 
\begin{align*}
f_t(-1-\delta) &= \frac{1}{2}(n-2t)^2 + \frac{1}{2} \delta^2 t^2 - \frac{1}{2} \delta^2 t - \delta n t + 2 \delta t^2 - \frac{3}{2} n + 2t + 1 \\
&= \frac{1}{2}(n-2t)^2 \pm C_7 n^{9/5}.
\end{align*}
This is \(\frac{1}{2}(n-2t)^2 + o(n^2)\), i.e., within \(o(n^2)\) of an eigenvalue of \(\frac{1}{2}\mathbb{J}\).  Furthermore, since there is some \(t \in \{0,\ldots,n/2\}\) for which this is true for any choice of eigenvalue \(s^2/2\), we have all elements of \(M\) are within \(o(n^2)\) of an element of \(L\).

Now, the discriminant of \(P_t(c)\) (with \(t=\alpha n\)) is
\[
\Delta = -\frac{1}{4} ((\alpha - 1) \alpha (4 n - 9) n^2 + (n-3)^3) (2 \alpha - 1)^2 (\alpha - 1) \alpha (n - 2)^2 n^6.
\]
Since, for \(\alpha \in (0,1/2)\), we have \((2\alpha - 1)^2 > 0\), \(\alpha - 1 < 0\), and \(\alpha > 0\), it follows that \(\Delta\) has the same sign as \((\alpha - 1) \alpha (4 n - 9) n^2 + (n-3)^3\).  For \(n > 3\), this is positive if
\[
\alpha < \frac{1}{2} - \frac{n-2}{2n} \sqrt{\frac{27}{4n - 9}} =: \frac{1}{2} - \eta .
\]
Since \(P_t(c)\) has at least one real root, if \(0 < \alpha < 1/2-\eta\), then all four roots of \(P_t(c)\) are real.  If \(1/2 > \alpha > 1/2 - \eta\), then \(P_t(c)\) has two real roots as described above, and another two conjugate complex roots.

In the regime where all roots of \(P_t(c)\) are real, we need only concern ourselves with negative values of \(f_t(c)\), since the spectral radius of \(K^{(3)}_n\) is \(\binom{n-1}{2}\), so any real nonnegative element of \(L\) lies in \([0,\binom{n-1}{2}]\), i.e., within \(o(n^2)\) (indeed, \(O(n)\)) of an element of \(M\).  If \(t \leq \sqrt{n}+1\), then \(f_t(c) < 0\) only when
\[
c < \frac{1-n+t+\sqrt{\frac{(n-t-1)(n-2)}{t}}}{t-1}.
\]
However,
\[
\frac{1-n+t+\sqrt{\frac{(n-t-1)(n-2)}{t}}}{t-1} \leq -\sqrt{n}+1+n^{1/4} < -\frac{1}{2} \sqrt{n}
\]
for sufficiently large \(n\).  Then, since \(P_t = c^2 f_t - g_t\) with \(g_t = \binom{t-1}{2}c^2 + (t-1)(n-t)c + \binom{n-t}{2}\), we have that \(P_t(c)=0\) implies
\begin{align*}
f_t(c) &= \frac{g_t(c)}{c^2} = \frac{\binom{t-1}{2}c^2 + (t-1)(n-t)c + \binom{n-t}{2}}{c^2} \\
&= \binom{t-1}{2} + \frac{(t-1)(n-t)}{c} + \frac{\binom{n-t}{2}}{c^2} \\
&> \binom{t-1}{2} - \frac{2(t-1)(n-t)}{\sqrt{n}} + \frac{\binom{n-t}{2}}{n} > -n^{3/2},
\end{align*}
so \(f_t(c) = o(n^2)\).  On the other hand, if \(t > \sqrt{n}+1\), then
\[
\min_c f_t(c) = \frac{n^2 - (n - 2) t - 3 n + 2}{2(1-t)} > -n^{3/2},
\]
which is again \(o(n^2)\).

It is straightforward (though laborious -- we employed SageMath \cite{sagemath}) to check using Cardano's Formula that, setting \(t = n/2 - \eta n\) for some \(\eta < (n-2) \sqrt{27/(n-9)}/(2n)\), the two complex roots of \(P_t(c)/(c-1)\) are \(c = -1 \pm C_8 n^{-1/2}\) for two conjugate choices of constants \(C_8\), and that plugging such values into \(f_t(c)\) yields \(O(n)\).  Therefore, every element of \(L\) is within \(o(n^2)\) of some element of \(M\), and the result follows.

\end{proof}
\section{Random Hypergraphs}
\label{sec:random}

Let \(D(n,k,p)\) denote the difference between the appropriately scaled all-ones hypermatrix and the adjacency matrix of the random \(k\)-uniform hypergraph with probability \(p \in (0,1)\), i.e.,
\[
D(n,k,p) = p \mathbb{J}_n^k - (k-1)! \mathcal{A}(G_k(n,p)).
\]
The special case of \(p = 1/2\) is of particular interest, since then \(2 D(n,k,1/2)\) is the random symmetric sign hyperensemble, i.e., an order \(k\), dimension \(n\) hypermatrix whose entries \(d_{\bi}\), where \(\bi = (i_1,\ldots,i_k)\) is a multi-index, are iid \(\pm 1\) with probability \(1/2\) if \(i_1 < \cdots < i_k\) and determined by symmetry for all other entries.  Let \(U(n,k,p)\) denote a random ``upper symplectic'' \(\{p,p-1\}\)-hypermatrix, i.e., an order \(k\), dimension \(n\) hypermatrix whose entries \(t^\prime_{\bi}\), where \(\bi = (i_1, \ldots,i_k)\) is a multi-index, are independently \(p\) with probability \(1-p\) or \(p-1\) with probability \(p\) if \(i_1 < \cdots < i_k\), and \(0\) otherwise.

For a vector \(\bv \in \CC^N\), denote by \(\|\bv\|_k\) the usual \(k\)-norm, i.e.,
\[
\|\bv\|_k = \left ( \sum_{j=1}^N |v_j|^k \right )^{1/k}.
\]



\begin{prop} Let \(\bv \in \CC^N\) be such that \(\|\bv\|_2 = 1\).  Then
\[
\Pr \left [ \|U : \bv^{\otimes (k-1)} \|_2 \geq t \right ] \leq C e^{-c t^2}
\]
for some absolute constants \(c,C > 0\), where \(U = U(t,k,p)\).
\end{prop}
\begin{proof} Let \(U_i\) denote the order \(k-1\) hypermatrix obtained by fixing the first coordinate to \(i\), i.e.,
\[
U_i = U : \be_i,
\]
where \(\be_i\) is the \(i\)-th elementary vector.  By the Azuma-Hoeffding inequality,
\begin{align*}
\Pr \left [ |U_i : \bv^{\otimes (k-1)} | \geq s \right ] &\leq 2 \exp \left ( \frac{-s^2}{2\|\bv^{\otimes (k-1)}\|_2^2} \right ) \\
&= 2 \exp \left ( \frac{-s^2}{2\|\bv\|_2^{2(k-1)}} \right ) \\
&= 2 e^{-s^2/2}.
\end{align*}
Therefore, integrating by parts,
\begin{align*}
\bE \left [ \exp(|U_i : \bv^{\otimes (k-1)} |^2/4) \right ] & \leq 1 + \int_{0}^\infty \frac{s}{2} \cdot e^{s^2/4} \Pr \left [ |U_i : \bv^{\otimes (k-1)} | \geq s \right ] \, ds \\
& \leq 1 + \int_{0}^\infty s e^{-s^2/4} \, ds \\
&= C
\end{align*}
for some \(C > 0\).  Then,
\begin{align*}
\bE \left [ \exp(\|U: \bv^{\otimes (k-1)} \|_2^2/4) \right ] &= \bE \left [ \exp(\sum_i |U_i : \bv^{\otimes (k-1)} |^2/4) \right ] \\
&\leq \prod_i \bE \left [ \exp( |U_i : \bv^{\otimes (k-1)} |^2/4) \right ] \\
&\leq C^n,
\end{align*}
where the second first inequality follows from independence.  Then, setting \(A = \sqrt{\log C / 8}\), we have
\begin{align*}
\Pr \left [ \|U : \bv^{\otimes (k-1)} \|_2 \geq A \sqrt{n} \right ] &= \Pr \left [ \exp(\|U : \bv^{\otimes (k-1)} \|_2^2/4) \geq \exp(2 n \log C) \right ] \\
&\leq \frac{\bE \left [ \exp(\sum_i |U_i : \bv^{\otimes (k-1)} |^2/4) \right ]}{\exp(2 n \log C)} \\
&\leq C^n e^{-2n \log C} = e^{- n \log C},
\end{align*}
by Markov's Inequality.
\end{proof}

\begin{cor} \label{cor:onevec} Let \(D = D(n,k,p)\) for \(p \in (0,1)\). Let \(\bv \in \CC^N\) be such that \(\|\bv\|_2 = 1\).  Then
\[
\Pr \left [ \|D : \bv^{\otimes (k-1)} \|_2 \geq t \right ] \leq C e^{-c t^2}
\]
for some constants \(c,C > 0\) (possibly depending on \(k\)), where \(D = D(n,k,p)\).
\end{cor}
\begin{proof}
For a permutation \(\sigma : [k] \rightarrow [k]\) and an order \(k\), dimension \(n\) cubical hypermatrix \(M\) whose \(\bi = (i_1,\ldots,i_k)\) multiindexed entry is \(m_{\bi}\), write \(M^\sigma\) for the order \(k\), dimension \(n\) cubical hypermatrix whose \(\bi^\sigma = (i_{\sigma(1)},\ldots,i_{\sigma(k)})\) entry is \(m_{\bi^\sigma}\).   By the triangle inequality,
\[
\|D : \bv^{\otimes (k-1)} \|_2 \leq \sum_{\sigma \in \mathfrak{S}_k} \|U^\sigma : \bv^{\otimes (k-1)} \|_2
\]
whence
\begin{align*}
\Pr \left [ \|D : \bv^{\otimes (k-1)} \|_2 \geq t \right ] &\leq k! \Pr \left [ \|U : \bv^{\otimes (k-1)} \|_2 \geq \frac{t}{k!} \right ] \\
& \leq C k! e^{-c (t/k!)^2} = C^\prime e^{-c^\prime t^2}.
\end{align*}
\end{proof}

\begin{prop} We have
\[
\Pr \left [ \exists \bv \in \CC^n : \|\bv\|_2 = 1 \textrm{ and } \|D : \bv^{\otimes (k-1)} \|_2 \geq B \sqrt{n \log n} \right ] = o(1).
\]
for some constant \(B\) depending on \(k\).
\end{prop}
\begin{proof} Let \(\delta = 1/(4k\sqrt{n})\) and let \(V\) denote the set of vectors \(\{\delta(s + s^\prime \sqrt{-1}) : -\delta^{-1} \leq s \leq \delta^{-1}, -\delta^{-1} \leq s^\prime \leq \delta^{-1} , s \in \ZZ, s^\prime \in \ZZ \}^n\) and
\[
V^\prime = \left \{ \frac{\bw}{\|\bw\|_2} : \bw \in V \right \},
\]
so that \(|V^\prime| \leq (5\delta)^{-2n}\).  Then, for any \(\bv \in \CC^N\) with \(\|\bv\|_2 = 1\), there is a \(\bw \in V^\prime\) so that, for each \(j \in [k]\),
\[
\max\{\Re(|\bv_j-\bw_j|),\Im(|\bv_j-\bw_j|)\} \leq \delta.
\]
Hence, letting \(\epsilon = \bv - \bw\), applying symmetry and the triangle inequality,
\begin{align*}
\|D : \bv^{\otimes (k-1)} \|_2 &= \|D : (\bw + \epsilon)^{\otimes (k-1)} \|_2 \\
&= \left \|D : (\bw^{\otimes (k-1)} + (k-1) \epsilon \otimes \bw^{\otimes (k-2)} \right . \\
& \qquad \left . + \binom{k-1}{2} \epsilon^{\otimes 2} \otimes \bw^{\otimes (k-3)} + \cdots ) \right \|_2 \\
&\leq \left \|D : \bw^{\otimes (k-1)} \right \|_2 + (k-1) \left \| \epsilon \|_2 \| \bw \right \|_2^{k-2} \\
& \qquad + \binom{k-1}{2} \left \| \epsilon \|_2^2 \| \bw \right \|_2^{k-3} + \cdots \\
&\leq \left \|D : \bw^{\otimes (k-1)} \right \|_2 + 2(k-1) \delta n^{1/2} \\
& \qquad + 4 \binom{k-1}{2} \delta^2 n + 8 \binom{k-1}{3} \delta^3 n^{3/2} + \cdots \\
&\leq \left \|D : \bw^{\otimes (k-1)} \right \|_2 + 2(k-1) \delta n^{1/2} \left ( 1 + 2(k-1)\delta n^{1/2} \right . \\
& \qquad \left . + 4(k-1)^2 \delta^2n + \cdots \right ) \\
&= \left \|D : \bw^{\otimes (k-1)} \right \|_2 + 2(k-1) \delta n^{1/2} \cdot \frac{1}{1-2(k-1)\delta n^{1/2}} \\
&\leq \left \|D : \bw^{\otimes (k-1)} \right \|_2 + 4(k-1) \delta n^{1/2} \\
&\leq \left \|D : \bw^{\otimes (k-1)} \right \|_2 + 1.
\end{align*}
Therefore, the the union bound and Corollary \ref{cor:onevec}, the probability that there exists a \(\bv \in \CC^n\) so that \(\|\bv\| = 1\) and \(\|D : \bv^{\otimes(k-1)}\| > B \sqrt{n 
\log n}\) is at most
\begin{align*} 
\sum_{w \in V^\prime} \Pr \left [ \left \|D : \bw^{\otimes (k-1)} \right \|_2 \geq B \sqrt{n \log n} \right ] & \leq (5 \delta)^{-2n} C e^{-c B^2 n \log n} \\
& \leq (400k^2)^{n} C n^{n(1-c B^2)}= o(1)
\end{align*}
for some constant \(B\) depending on \(k\).
\end{proof}

\begin{theorem} \label{thm:randomradius} For \(D = D(n,k,p)\), \(\rho(D) < B n^{(k-1)/2} \sqrt{\log n}\) with high probability for some constant \(B\) depending on \(k\).
\end{theorem}
\begin{proof} Suppose
\[
D : \bv^{\otimes k-1} = \lambda \bv^{\circ k-1}.
\]
Since this equation is homogeneous, we assume that \(\|\bv\|_2 = 1\) by scaling \(\bv\).  By H\"older's Inequality,
\begin{align*}
\left \|\bv^{\circ k-1} \right \|_2 &= \left \|\bv^{\circ 2} \right \|_{(k-1)}^{(k-1)/2} \\
&= \left (\left \|\bv^{\circ 2} \right \|_{k-1} \frac{\left \|\hat{1} \right \|_{(k-1)/(k-2)}}{\left \|\hat{1} \right \|_{(k-1)/(k-2)}} \right )^{(k-1)/2} \\
&\geq \|\bv^{\circ 2}\|_1 n^{1-k/2} \\
&= \|\bv\|_2^2 n^{1-k/2} = n^{1-k/2}.
\end{align*}
Hence,
\begin{align*}
\left \| D : \bv^{\otimes k-1} \right \|_2 &= |\lambda| \left \| \bv^{\circ k-1} \right \|_2 \\
&\geq |\lambda| n^{1-k/2}.
\end{align*}
By the previous result, with high probability this equation holds for every such \(\lambda\) as long as the eigenvalue of largest modulus satisfies
\[
|\lambda| n^{1-k/2} < B \sqrt{n \log n},
\]
i.e.,
\[
|\lambda| < B n^{(k-1)/2} \sqrt{\log n}.
\]
\end{proof}

Note that this bound is not quite tight for \(k=2\).  Indeed, a famous result of F\"{u}redi and Koml\'{o}s (\cite{FurediKomlos}) is that the random symmetric sign ensemble has spectral radius \(O(\sqrt{n})\) with high probability, instead of \(O(\sqrt{n \log n})\), as the above result implies.  Therefore, the question remains of whether Theorem \ref{thm:randomradius} is tight for \(k > 2\).

\section{Piecewise Smoothness of Eigenvalues}
\label{sec:smootheigs}

Suppose \(A = A(t)\) is a real-analytic function in \(t \in [0,1]\) of order \(k\), dimension \(n\) real symmetric hypermatrices.  Then, by \cite{MR2178089}, \(f(q,t) = \det(A(t)-qI)\) is a monic polynomial in \(q\) of degree \(n(k-1)^{n-1}\) with real-analytic coefficients, polynomials in the entries of \(A(t)\).  For any fixed \(t \in [0,1]\), we may factor
\[
f(q,t) = \prod_{i=1}^{r(t)} (q - \lambda_i(t))
\]
where \(\{\lambda_i\}_{i=1}^{r(t)}\) are the distinct roots of \(f(q,t)\).  For each positive integer \(m\), we define the \(m\)-th generalized discriminant by
\[
\Delta_m f(q,t) = \sum_{S \in \binom{[r(t)]}{m}} \prod_{\substack{i, j \in S \\ i \neq j}} \left (\lambda_{i} - \lambda_{j} \right)^2.
\]
(When \(m=1\), \(\Delta_1 f(q,t) = n(k-1)^{n-1}\).)  By \cite{MR2806697}, the number of distinct \(q\)-roots of \(f(q,t)\) is the largest \(m\) so that \(\Delta_m f(q,t) \neq 0\).  Since \(\Delta_m f(q,t)\) is a polynomial whose coefficients are polynomials in the coefficients of \(f\), its coefficients are also real-analytic, so it is identically zero or its roots are discrete (and therefore finite in number in \([0,1]\)).  Let \(R\) be the largest integer so that it is not identically zero.  Then \(r(t) = R\) for all but a finite set of points 
\(E\), and we may write 
\[
F(q,t) = \Delta_R f(q,t) = \sum_{S \in \binom{[R]}{m}} \prod_{\substack{i, j \in S \\ i \neq j}} \left (\lambda_{i} - \lambda_{j} \right)^2,
\]
a real-analytic, nonzero function on all but \(E^\prime = [0,1] \setminus E\).  Then, by Lemma 3.2 of \cite{MR2806697}, we can factor \(F(q,t)\) as \(F_1(q,t) \cdots F_R(q,t)\) in an open interval \(I\) about any point \(t \in E^\prime\) so that the roots (in \(q\)) of \(F_i(q,t)\) are a single \(C^\infty\) function \(\lambda^\prime_i(t)\).  It is straightforward to see that (by real-analytic continuation), because \(E\) is discrete, this implies that the \(q\)-degree of each of the \(F_i(q,t)\) is constant on each component of \(E^\prime\), and, furthermore, that the roots of \(f(q,t)\) can be taken to be smooth on each component of \(E^\prime\).

We may conclude the following.

\begin{theorem} Suppose that  \(A = A(t)\) is a real-analytic function in \(t \in [0,1]\) of order \(k\), dimension \(n\) real symmetric hypermatrices.  Then there are functions \(\lambda_i(t)\), \(i \in [n (m-1)^{n-1}]\), so that the \(\lambda_i(t)\) are precisely the eigenvalues of \(A(t)\), and each \(\lambda_i(t)\) is smooth except at finitely many points. 
\end{theorem}

\section{Acknowledgments}
Thank you to Alexander Auerback, Liqun Qi, Vladimir Nikiforov, Armin Rainer, and Jose Israel Rodriguez for helpful discussions during the development of the present manuscript.

\bibliographystyle{siam}
\bibliography{hmp-refs}

\end{document}